\documentclass[twoside,a4paper,12pt]{article}
\usepackage{amssymb, amsmath, amsthm}
\usepackage{color}
\usepackage[utf8]{inputenc}
\usepackage{hyperref}

\newdimen\myMargin
\textwidth \paperwidth
\advance\textwidth -2\myMargin
\textheight \paperheight
\advance\textheight -2\myMargin
\advance\oddsidemargin \myMargin
\advance\evensidemargin \myMargin
\advance\topmargin \myMargin
\advance\textheight -17 true mm
\advance\topmargin 12.5 true mm

\newcommand{\A}{{\cal A}}
\def\B{{\cal B}}
\def\Ha{{\cal H}}

\newcommand{\Q}{\mathbb Q}

\newcommand{\F}{\mathbb F}

\newtheorem{fact}{Fact}
\newtheorem{theorem}{Theorem}

\newtheorem{corollary}[theorem]{Corollary}
\newtheorem{proposition}[theorem]{Proposition}
\newtheorem{definition}[theorem]{Definition}

\newtheorem{algorithm}{Algorithm}

\theoremstyle{remark}\newtheorem{example}[theorem]{Example}
\theoremstyle{remark}\newtheorem{remark}[theorem]{Remark}

\title{Splitting quaternion algebras over quadratic number fields}
\author{\normalsize
 \begin{minipage}{0.3\linewidth}
    \large
    P\'eter Kutas \\
    \footnotesize
Institute for Computer Science and Control,
Hungarian Acad. Sci. and
Department of Mathematics and its Applications,
Central European University
    \texttt{Kutas\_Peter@phd.ceu.edu} \\
    \normalsize
  \end{minipage}
  }
\begin{document}
\maketitle
\begin{abstract}
We propose an algorithm for finding zero divisors in quaternion algebras
over quadratic number fields, or equivalently, solving homogeneous quadratic equations
in three variables over $\mathbb{Q}(\sqrt{d})$ where $d$ is a square-free integer. The algorithm
is randomized and runs in polynomial time if one is allowed to call oracles for
factoring integers.

\end{abstract}


\bigskip
\noindent
{\textbf{Mathematics Subject Classification:} 68W30, 16Z05, 11D09

\section{Introduction}

In this note we consider the following algorithmic problem which we call explicit isomorphism problem: let $K$ be a field and let $\A$ be a $K$-algebra isomorphic to $M_n(K)$ given by a collection of structure constants (i.e. via its regular representation). The task is to construct an explicit isomorphism between $\A$ and $M_n(K)$ or, equivalently, to find a primitive idempotent in $\A$.

Although the problem comes from computational representation theory, it has various applications in computational algebraic geometry and number theory as well. The case where $K=\mathbb{Q}$, has connections with explicit $n$-descent on elliptic curves \cite{Cremona3}, solving norm equations \cite{IRS} and parametrizing Severi-Brauer surfaces \cite{de Graaf}. In \cite{IKR} we consider the case where $K=\mathbb{F}_q(t)$ which is connected to the factorization problem in a certain skew-polynomial ring \cite{GZ03},\cite{GTLN}.

Ivanyos, R\'onyai and Schicho proposed an ff-algorithm for the case where $K$ is an algebraic number field \cite{IRS}. An ff-algorithm is allowed to call an oracle for factoring an integer or polynomial over a finite field at a cost of the size of the input to the oracle call. An ff-algorithm can also be turned into a randomized polynomial-time algorithm (of Las Vegas type) which is allowed to call an oracle for factoring integers (as a polynomial over a finite field can be factored by a randomized polynomial-time algorithm \cite{Berlekamp}). It is natural to consider ff-algorithms for this task as Rónyai showed that the problem of computing an explicit isomorphism between $\A$ and $M_2(\Q)$ is at least as hard as factoring integers \cite{Ronyai}. The algorithm from \cite{IRS} however depends exponentially on the degree of the number field, the dimension of the matrix algebra and the logarithm of the discriminant of the number field. This algorithm was improved in \cite{ILR}.

The first result for number fields with non-bounded discriminant is contained in \cite{MACIS}. In that extended abstract the following problem is addressed. Let $K=\mathbb{Q}(\sqrt{d})$ and let $\A$ be isomorphic to $M_2(\mathbb{Q}(\sqrt{d}))$ given by structure constants. Then a randomized polynomial-time algorithm which is allowed to call oracles for factoring integers is proposed for finding a quaternion $\Q$-subalgebra $\B$ of $\A$. This does not solve the explicit isomorphism problem, as it may occur that $\B$ is a division algebra and therefore contains no zero divisors. Thus the problem of finding a zero divisor in $\A$ was left open. 

This note is a completion of \cite{MACIS}. We propose a randomized polynomial-time algorithm of Las Vegas type which uses an oracle for integer factorization to compute an explicit isomorphism when $\A\cong M_2(\Q(\sqrt{d}))$. We use the method from \cite{MACIS} to construct a quaternion $\Q$-subalgebra $\B$. The key observation is that $\B$ is split by $\mathbb{Q}(\sqrt{d})$, therefore contains $\mathbb{Q}(\sqrt{d})$ as a subfield. Specifically, it contains an element $s$ which is not in the center of $\A$ and $s^2=d$. Finally in Theorem \ref{main} we show how to find such an element $s$ and output the zero divisor $s-\sqrt{d}$. Note that from a zero divisor $e$ an explicit isomorphism between $\A$ and $M_2(\mathbb{Q}(\sqrt{d}))$ can be constructed by a standard procedure, by considering the left action of $\A$ (the action is multiplication from the left) on the minimal left ideal generated by $e$.

All our main algorithms rely on finding nontrivial zeros of quadratic forms in several variables over $\mathbb{Q}$. In Section 2 we provide a brief summary of the running times of these previously known algorithms and we give a general introduction on quaternion algebras. In Section 3 we describe our main algorithms. Some of the results (Proposition \ref{anticommute}, \ref{square}) are already contained in the extended abstract \cite{MACIS}. We have also implemented the main algorithms in MAGMA \cite{MAGMA}. The program code is available on the author's webpage (\url{https://sites.google.com/site/kutasp89/thesis}) and a description of the implementation can be found in the PhD thesis of the author \cite[Section 6.2.]{Phd}.


\section{Quaternion algebras}
\subsection{General properties}
In this subsection we recall some basic facts about quaternion algebras. All these facts can be found in \cite{Vigneras}.

\begin{definition}
Let $K$ be a field. A central simple algebra $\A$ over $K$ is called a \textbf{quaternion algebra} if it has dimension 4 over $K$.
\end{definition}

A quaternion algebra has a special $K$-basis as stated below:
\begin{proposition}
Let $char(K)\neq 2$ and let $H$ be a quaternion algebra over $K$. Then $H$ has a $K$-basis $1,u,v,uv$ such that $uv=-vu$ and $u^2$ and $v^2$ are in the center of $H$. We call such a basis a quaternion basis of $H$.
\end{proposition}
\begin{remark}
This result is well known, a proof can be found in \cite{Vigneras}. There is a similar presentation if $char(K)=2$, however since we will later only consider algebraic number fields, we omit this statement here.
\end{remark}
From now on we assume that $char(K)\neq 2$. Since the center of $H$ is $K$, we have that $u^2\in K$ and $v^2\in K$ if we identify 1 with the identity element of $H$. This motivates the following notation:
\begin{definition}
Let $H$ be a quaternion algebra over $K$ with quaternion basis $1,u,v,uv$. Let $u^2=\alpha$ and $v^2=\beta$. Note that $\alpha$ and $\beta$ are in $K$. Then we denote $H$ by $H_K(\alpha,\beta)$.
\end{definition}

It is easy to see that this is well-defined, i.e. all quaternion algebras which have a quaternion basis $1,u,v,uv$ such that $u^2=\alpha$ and $v^2=\beta$ are isomorphic.

The Wedderburn-Artin theorem implies that every quaternion algebra is either isomorphic to $M_2(K)$ or is a division algebra over $K$. There is a nice criterion which tells us when a quaternion algebra is split (i.e., is isomorphic to $M_2(K)$). First we recall some definitions.
\begin{definition}
Let $H$ be a quaternion algebra over $K$, with quaternion basis $1,u,v,uv$. Let $s=\lambda_1+\lambda_2u+\lambda_3v+\lambda_4uv$. Then let $\sigma(s)=\lambda_1-\lambda_2u-\lambda_3v-\lambda_4uv$ be the \textbf{conjugate} of $s$. We call $Tr(s)=s+\sigma(s)$ the \textbf{trace} of $s$ and $N(s)=s\sigma(s)$ the \textbf{norm} of $s$. Note that both $Tr(s)$ and $N(s)$ are in $K$.
\end{definition}
\begin{remark}
One can show that the functions $Tr(x)$ and $N(x)$ do not depend on the quaternion basis and coincide with the usual reduced trace and reduced norm (see \cite{Vigneras}).
\end{remark}

\begin{proposition}\label{split}
The following statements are equivalent:
\begin{enumerate}
\item $H_K(\alpha,\beta)\cong M_2(K)$,
\item There exists a nonzero element $s\in H_K(\alpha,\beta)$ such that $N(s)=0$,
\item The quadratic form $x_1^2-\alpha x_2^2-\beta x_3^2+\alpha\beta x_4^2$ is isotropic over $K$,
\item There exists a nonzero element $s\in H_K(\alpha,\beta)$ such that $Tr(s)=0$ and $N(s)=0$,
\item The quadratic form $\alpha x^2+\beta y^2-z^2$ is isotropic over $K$.
\end{enumerate}
\end{proposition}

\begin{remark}
If we write out condition (2) in terms of the quaternion basis we obtain (3). This shows that if
$$x_1^2-\alpha x_2^2-\beta x_3^2+\alpha\beta x_4^2=0, $$
then $1+x_1u+x_2v+x_3uv$ is a zero divisor (or equivalently has norm zero) in $H_K(\alpha,\beta)$.  
Condition (4) if written out would give the equation $\alpha x_0^2+\beta y_0^2-\alpha\beta z_0^2=0$ (since every element $x$ for which $Tr(x)=0$ is the linear combination of $u,v$ and $uv$). By a change of variables we arrive at (5) ($x:=\frac{y_0}{\alpha},y:=\frac{x_0}{\beta}, z:=z_0$). Thus from a solution to the equation 
$$\alpha x^2+\beta y^2-z^2=0,$$
a zero divsor in $H_K(\alpha,\beta)$ can be obtained by a polynomial-time algorithm (first reverse the change of variables and then apply condition (3) as discussed above). 

Details can be found in \cite{Vigneras} (or \cite{Castel},\cite{Ronyai}). Note that this shows that there is a strong connection between quaternion algebras and quadratic forms in three variables over $K$.
\end{remark}

\subsection{Algorithmic results}

Now we review some algorithmic results concerning quaternion algebras and quadratic forms over $\mathbb{Q}$. Note that we only consider deterministic and randomized polynomial-time algorithms (which may use oracles for factoring integers). In this note every randomized algorithm is of Las Vegas type.  
In this section we consider an algebra to be given as a collection of structure constants, which means the following.
Let $\A$ be an algebra over the field $K$. Let $a_1,\dots, a_m$ be a $K$-basis of $\A$. Then the products of the basis elements can be expressed as the $K$-linear combination of the basis elements:
$$ a_ia_j=\gamma _{ij1}a_1+\gamma _{ij2}a_2+\cdots +\gamma _{ijm}a_m. $$

The $\gamma_{ijk}\in K$ are called structure constants. Note that specifying $\A$ with structure constants is equivalent to giving $\A$ by its regular representation.
\begin{example}
Let $H_K(\alpha,\beta)$ be a quaternion algebra with the quaternion basis $1,u,v,uv$. Then every basis element is given by a $4\times 4$ matrix:
$$
\begin{pmatrix}
		1 & 0 & 0 & 0\\
		0 & 1 & 0 & 0 \\
        0 & 0 & 1 & 0 \\
        0 & 0 & 0 & 1
\end{pmatrix}
\begin{pmatrix}
		0 & 1 & 0 & 0\\
		\alpha & 0 & 0 & 0 \\
        0 & 0 & 0 & 1 \\
        0 & 0 & \alpha & 0
\end{pmatrix}
\begin{pmatrix}
		0 & 0 & 1 & 0\\
		0 & 0 & 0 & -1 \\
        \beta & 0 & 0 & 0 \\
        0 & -\beta & 0 & 0
\end{pmatrix}
\begin{pmatrix}
		0 & 0 & 0 & 1\\
		0 & 0 & -\alpha & 0 \\
        0 & \beta & 0 & 0 \\
        -\alpha\beta & 0 & 0 & 0
\end{pmatrix}.$$

\end{example}
\smallskip

R\'onyai \cite{Ronyai} gave a polynomial-time algorithm for finding a quaternion representation from an arbitrary structure constant representation. Thus we may assume that a quaternion algebra is given by a quaternion basis.

\begin{definition}
Let $\A\cong M_n(K)$ be given by structure constants. The \textbf{explicit isomorphism} problem is to compute an isomorphism between $A$ and $M_n(K)$.
\end{definition}
\begin{remark}\label{rank1}
Finding an explicit isomorphism is equivalent to finding an element $r$ of rank 1 in $\A$. Indeed, the left action of $A$ on the left ideal $Ar$ produces such an isomorphism (the vector space $Ar$ has dimension $n$ and the left action is $K$-linear so every element of $A$ can be represented by an $n\times n$ matrix and this map is an isomorphism).
\end{remark}

R\'onyai showed (\cite{Ronyai}) that there is a randomized polynomial-time reduction from factoring square-free integers to the explicit isomorphism problem in the case $\A\cong M_2(\mathbb{Q})$. Ivanyos and Sz\'ant\'o \cite{IS} proposed a polynomial-time ff-algorithm to solve this problem. They construct a maximal order (using the algorithm from \cite{RLIR}) and use lattice reduction to find a zero divisor. Note that an ff-algorithm can also be thought of as a randomized polynomial-time algorithm which is allowed to call oracles for factoring integers. 

Cremona and Rusin gave a different algorithm \cite{CR} for the same task which runs in polynomial time if one is allowed to call an oracle for integer factorization. They proposed an algorithm which finds nontrivial zeros of quadratic forms in three variables over $\Q$. From this data a zero divisor in $\A$ can be constructed via the description in Proposition \ref{split}. The algorithms from \cite{de Graaf},\cite{Pilnikova} and \cite{IRS} generalize these results to matrix algebras of higher degree.

However, if $K$ is a number field then the algorithms from \cite{IRS} and \cite{IS} run exponentially in the degree and the logarithm of the discriminant of the number field.

In the next section we consider the case where $\A\cong M_2(K)$ where $K$ is a quadratic extension of $\mathbb{Q}$. It turns out that this is related to finding nontrivial zeros of quadratic forms over $\mathbb{Q}$ in several variables. Hence we cite these two results:
\begin{fact}[Simon \cite{Simon}]
There is a randomized polynomial time-algorithm for finding nontrivial zeros (or proving that no such zero exists) of quadratic forms over $\mathbb{Q}$ in dimension at least 4 if one is allowed to call oracles for factoring integers.
\end{fact}

The paper of Simon \cite{Simon} was presented at the conference "Recent Developments in Computational Number Theory" (\url{http://poncelet.sciences.univ-metz.fr/~soriano/ProgrammeCIRM.pdf}) and is implemented in MAGMA \cite{MAGMA}. 
Finding nontrivial zeros of quadratic forms in 4 variables over $\mathbb{Q}$ is also at least as hard as factoring integers since quadratic forms in dimension 4 with square discriminant correspond to quadratic forms of dimension 3 (see \cite{Castel}).
Castel \cite{Castel} improved these algorithms and obtained an algorithm which works in dimension 5 (and above) and does not depend on factoring integers. However, its running time calculations depend on the validity of the Generalized Riemann Hypothesis (GRH).
\begin{fact}[Castel \cite{Castel}]
Assuming GRH, there is a randomized polynomial-time algorithm which finds a nontrivial zero of an indefinite quadratic form (over $\mathbb{Q}$) in dimension 5 (or more).
\end{fact}

\section{Finding a zero divisor}
In this section we propose an algorithm for finding a zero divisor in $A$ which is isomorphic to $M_2(\mathbb{Q}(\sqrt{d}))$ and is given by structure constants. First we construct a subalgebra $B$ in $A$ which is a quaternion algebra over $\mathbb{Q}$. Then, with this information at our hands, we construct a zero divisor. In Remark \ref{rank1} we saw how to construct an explicit isomorphism from a zero divisor. First we outline the steps of our algorithm:

\smallskip
\begin{algorithm}\label{alg}
\mbox{~}

\begin{enumerate}
\item Find an element $u\in \A$ such that $Tr(u)=0$ and $u^2\in\mathbb{Q}$ and $u\neq 0$.
\item Find a nonzero element $v$ such that $uv=-vu$ and $v^2\in\mathbb{Q}$.
\item Let $B$ be the $\mathbb{Q}$-subspace generated by $1,u,v,uv$. $B$ is a quaternion algebra over $\mathbb{Q}$. Use the algorithm from \cite{IRS} (or \cite{IS}) to either find a zero divisor in $B$ or conclude that $B$ is a division algebra.
\item If $B$ is a division algebra then find an element $s\in B$ such that $s^2=d$. Return $s-\sqrt{d}$.
\end{enumerate}
\end{algorithm}

The key to each step is finding an isotropic vector for a quadratic form in several variables. In Step 1 we solve a homogeneous quadratic equation in 6 variables, in Step 2 and 3 an equation in 3 variables and finally in Step 4 an equation in 4 variables. Step 1,2 and 3 are already exhibited in \cite{MACIS}. Step 4 is the crucial new step which allows us to find a zero divisor in $\A$ and not just a quaternion subalgebra over $\Q$. Now we proceed by providing an algorithm for each step.

\begin{proposition}\label{square}
Let $\A\cong M_2(\mathbb{Q}\sqrt{d})$ be given by structure constants. Then there exists a randomized polynomial-time algorithm which is allowed to call an oracle for integer factorization which finds a nonzero $l\in\A$ for which $Tr(l)=0$ and $l^2\in\mathbb{Q}$.
\end{proposition}
\begin{proof}
 First we construct a quaternion basis $1,w,w',ww'$ of $\A$. We have the following:
\begin{equation*}
w^2=r_1+t_1\sqrt{d}, ~w'^2=r_2+t_2\sqrt{d}
\end{equation*}

If $t_1$ or $t_2$ is 0 then $w$ or $w'$ will be a suitable element. If $r_1t_2+r_2t_1=0$ then $(ww')^2\in \mathbb{Q}$ and is traceless. From now on we assume that $t_1,t_2$ and $r_1t_2+r_2t_1$ are nonzero.

Every element whose trace is 0 is in the $\mathbb{Q}(\sqrt{d})$-subspace generated by $w$, $w'$ and $ww'$. The condition $l^2\in\mathbb{Q}$ gives the following equation ($s_1,\dots,s_6\in\mathbb{Q}$):
\begin{equation*}
((s_1+s_2\sqrt{d})w+(s_3+s_4\sqrt{d})w'+(s_5+s_6\sqrt{d})ww')^2\in \mathbb{Q}
\end{equation*}

If we expand this we obtain:
\begin{eqnarray*}
((s_1+s_2\sqrt{d})w+(s_3+s_4\sqrt{d})w'+(s_5+s_6\sqrt{d})ww')^2=\\
(s_1^2+ds_2^2+2s_1s_2\sqrt{d})(r_1+t_1\sqrt{d})+(s_3^2+ds_4^2+2s_3s_4\sqrt{d})(r_2+t_2\sqrt{d})-\\
(s_5^2+ds_6^2+2s_5s_6\sqrt{d})(r_1+t_1\sqrt{d})(r_2+t_2\sqrt{d})
\end{eqnarray*}

In order for this to be in $\mathbb{Q}$ the coefficient of $\sqrt{d}$ has to be zero:
\begin{eqnarray}\label{eq}
t_1s_1^2+t_1ds_2^2+2r_1s_1s_2+t_2s_3^2+t_2ds_4^2+2r_2s_3s_4-(r_1t_2+t_1r_2)s_5^2-\\
(r_1t_2+t_1r_2)ds_6^2-2(r_1r_2+t_1t_2d)s_5s_6=0
\end{eqnarray}

The left hand side of Equation \ref{eq} is a quadratic form in the variables $s_1,\dots,s_6$. This implies that if it is indefinite then it has a solution. The Gram-matrix of the quadratic form is the following:
$$\begin{pmatrix}
t_1& r_1& 0& 0& 0& 0\\
r_1& t_2d&0& 0& 0& 0\\
0& 0& t_2& r_2& 0& 0\\
0& 0& r_2& t_2d& 0& 0\\
0& 0& 0& 0&-(r_1t_2+t_1r_2)&r_1r_2+t_1t_2d\\
0& 0& 0& 0&r_1r_2+t_1t_2d&-(r_1t_2+t_1r_2)d 
\end{pmatrix}$$

It is block diagonal with three $2\times 2$ blocks. The determinant of the first block is $t_1^2d-r_1^2$, the determinant of the second is $t_2^2d-r_2^2$ and the determinant of the third is $(r_1t_2+t_1r_2)^2d-(r_1r_2+t_1t_2d)^2$. Now we show that this quadratic form is always indefinite. If $d<0$ then $t_1^2d-r_1^2<0$ (it is nonzero since $t_1\neq 0$ and $d$ is a square-free integer), hence the form $t_1s_1^2+t_1ds_2^2+2r_1s_1s_2$ is indefinite. If $d>0$ then if either $t_1s_1^2+t_1ds_2^2+2r_1s_1s_2$ or $t_2s_3^2+t_2ds_4^2+2r_2s_3s_4$ is indefinite then we are done. So the remaining case is when $t_1^2d-r_1^2>0$ and $t_2^2d-r_2^2>0$. However, this implies that the quadratic form $-(r_1t_2+t_1r_2)s_5^2-(r_1t_2+t_1r_2)ds_6^2-2(r_1r_2+t_1t_2d)s_5s_6$ is indefinite since
$$(t_1^2d-r_1^2)(t_2^2d-r_2^2)=-((r_1t_2+t_1r_2)^2d-(r_1r_2+t_1t_2d)^2)$$.

Hence we have proven that the quadratic form
\begin{equation}
t_1s_1^2+t_1ds_2^2+2r_1s_1s_2+t_2s_3^2+t_2ds_4^2+2r_2s_3s_4-(r_1t_2+t_1r_2)s_5^2-(r_1t_2+t_1r_2)ds_6^2-2(r_1r_2+t_1t_2d)s_5s_6
\end{equation}
has a nontrivial zero over $\mathbb{Q}$. A nontrivial zero of the quadratic form in (1) can be found by Simon's algorithm \cite{Simon}. This is a randomized polynomial-time algorithm if one is allowed to call an oracle for factoring integers. 
\end{proof}
\begin{remark}
Observe that we only used the fact that $\A$ is a quaternion algebra over $\mathbb{Q}(\sqrt{d})$, we did not need the fact that it is in fact a full matrix algebra.
\end{remark}
\begin{remark} \label{random}
The main tool of this proof was an algorithm for finding nontrivial zeros of quadratic forms in 6 variables. For this task we also could have used Castel's algorithm \cite{Castel}. However, Castel's algorithm is dependent on GRH and we would like to have an algorithm which is independent of the validity of GRH.  
\end{remark}

We proceed to the next step:
\begin{proposition}\label{anticommute}
Let $\B=H_{\mathbb{Q}(\sqrt{d})}(a,b+c\sqrt{d})$ given by: $u^2=a, v^2=b+c\sqrt{d}$, where $a,b,c\in \mathbb{Q}, ~c\neq 0$. Then finding a nonzero element $v'$ such that $uv'+v'u=0$ and $v'^2$ is a rational multiple of the identity is polynomial-time equivalent to finding a zero divisor in the quaternion algebra $H_{\mathbb{Q}}((\frac{b}{c})^2-d,a)$.
\end{proposition}
\begin{remark}
By polynomial-time equivalent we mean the following. From a zero divisor in $H_{\mathbb{Q}}((\frac{b}{c})^2-d,a)$ a suitable element $v'\in\B$ can be constructed in polynomial time. On the other hand, from a suitable element $v'\in\B$ a zero divisor in $H_{\mathbb{Q}}((\frac{b}{c})^2-d,a)$ can be constructed in polynomial time as well. 
\end{remark}
\begin{proof}
Since $v'$ anticommutes with $u$ (i.e. $uv'+v'u=0$) it must be a $\mathbb{Q}(\sqrt{d})$-linear combination of $v$ and $uv$. This implies we have to search for $s_1,s_2,s_3,s_4\in\mathbb{Q}$ such that:
\begin{equation*}
((s_1+s_2\sqrt{d})v+(s_3+s_4\sqrt{d})uv)^2\in\mathbb{Q}
\end{equation*}

Expanding this expression we obtain the following:
\begin{eqnarray*}
((s_1+s_2\sqrt{d})v+(s_3+s_4\sqrt{d})uv)^2=\\
(s_1^2+s_2^2d+2s_1s_2\sqrt{d})(b+c\sqrt{d})-(s_3^2+s_4^2d+2s_3s_4\sqrt{d})a(b+c\sqrt{d})
\end{eqnarray*}

In order for this to be rational, the coefficient of $\sqrt{d}$ has to be zero. We obtain the following equation:

\begin{equation*}
c(s_1^2+s_2^2d)+2bs_1s_2-ac(s_3^2+s_4^2d)-2abs_3s_4=0
\end{equation*}

First we divide by $c$. Note that $c$ is nonzero. Let $f=b/c$.
\begin{equation}\label{notdiag}
s_1^2+s_2^2d+2fs_1s_2-a(s_3^2+s_4^2d)-2afs_3s_4=0
\end{equation}

First we diagonalize the left hand side of Equation \ref{notdiag}.  Consider the following change of variables: $x:=s_1+fs_2$, $y:=s_2$,$z:=s_3+s_4f$, $w:=s_4$. The transition matrix of this change of variables is the following:
$$\begin{pmatrix}
1& f&0 &0\\
0& 1& 0& 0\\
0& 0& 1& f\\
0&0&0&1
\end{pmatrix}.$$
The transition matrix is an upper triangular matrix with 1-s in the diagonal so it has determinant 1 (this means that these two quadratic forms are equivalent). In terms of these new variables the equation takes the following form:

\begin{equation*}
x^2+(d-f^2)y^2-az^2-a(d-f^2)w^2=0.
\end{equation*}

Finding a solution of this equation is polynomial-time equivalent to finding a zero divisor in the quaternion algebra $H_{\mathbb{Q}}(f^2-d,a)$ by Proposition \ref{split}.
\end{proof}

\begin{remark}
This statement can be interpreted constructively and as a complexity statement as well. First it provides a randomized polynomial-time algorithm (which is allowed to call an oracle for factoring integers) for finding such an element $v'$. It also says however, that finding such an element $v'$ is as hard as finding zero divisors in quaternion algebras over $\Q$. R\'onyai proved in \cite{Ronyai} that there is a randomized polynomial-time reduction from factoring integers to finding zero divisors in quaternion algebras over $\Q$. This implies that finding a quaternion subalgebra over $\mathbb{Q}$ containing $u$ is hard (otherwise one could easily find such an element $v'$). Also note that Simon's algorithm could also be applied to solving Equation \ref{notdiag} from which a suitable $v'$ can be constructed.
\end{remark}
\begin{remark}
Proposition \ref{anticommute} also provides the following result. Let $\B=H_{\Q(\sqrt{d})}(a,b+c\sqrt{d})$ where $a,b,c\in\Q$. Then $B$ contains a quaternion subalgebra over $\Q$ if and only if $H_{\mathbb{Q}}(b^2-cd^2,a)$ splits. The number $b^2-cd^2$ is the norm of $b+c\sqrt{d}$ in the extension $\Q(\sqrt{d})|\Q$. Actually $H_{\mathbb{Q}}(b^2-cd^2,a)$ is then the so-called corestriction of $\B$ \cite[Part II, Theorem 7]{Draxl}. It is known that if the corestriction of $\B$ splits then $\B$ contains a quaternion subalgebra over $\Q$, however, the usual proofs of this fact are not effective. For more details on the corestriction (or norm) of central simple algebras the reader is refferred to \cite{Draxl},\cite{Involutions}. 
\end{remark}



Finally putting Proposition \ref{square} and \ref{anticommute} together we obtain the following:
\begin{corollary}\label{quat}
Let $\A\cong M_2(\mathbb{Q}(\sqrt{d}))$ be given by structure constants. Then one can either find a zerod divisor in $\A$, or a four dimensional subalgebra over $\mathbb{Q}$ which is a quaternion algebra (and is split by $\mathbb{Q}(\sqrt{d})$) by a randomized polynomial-time algorithm which is allowed to call an oracle for factoring integers.
\end{corollary}
\begin{proof}
First we find a nonzero element $l$ such that $Tr(l)=0$ and $l^2\in\mathbb{Q}$ using the algorithm from Proposition \ref{square}. If $l^2=0$, then output $l$ as a zero divisor. If not, then we prove that there exists an element $l'$ such that $ll'+l'l=0$ and $l'^2\in\mathbb{Q}$.


If $l^2$ is a square in $\Q$, then such an $l'$ exists by Proposition \ref{anticommute}. Indeed let $l^2=c^2\in\Q$ and let $w$ be an element in $\A$ for which $wl=-lw$ and $w^2=e+f\sqrt{d}$. Then Proposition \ref{anticommute} asserts that a suitable $l'$ exists if and only if the quaternion algebra $H_{\Q}(d-\frac{e}{f}^2,c^2)$ splits. The quaternion algebra $H_{\Q}(d-\frac{e}{f}^2,c^2)$ does split since $c^2$ is a square in $\Q$ (thus $l-c$ is a zero divisor). 

From now on assume that $l^2$ is not a square in $\Q$. There exists a subalgebra $\A_0$ in $\A$ which is isomorphic to $M_2(\mathbb{Q})$. In this subalgebra there is an element $l_0$ for which $l$ and $l_0$ have the same minimal polynomial over $\mathbb{Q}(\sqrt{d})$. This means that there exists an $m\in \A$ such that $l=m^{-1}l_0m$ (\cite[Theorem 2.1.]{Vigneras}). There exists a nonzero $l_0'\in\A_0$ such that $l_0l_0'+l_0'l_0=0$. Let $l'=m^{-1}l_0'm$. We have that $l'^2=m^{-1}l_0'mm^{-1}l_0m=m^{-1}l_0^2m=l_0^2$, hence $l'^2\in\mathbb{Q}$. Since conjugation by $m$ is an automorphism we have that $ll'+l'l=m^{-1}(l_0l_0'+l_0'l_0)m=m^{-1}0m=0$. Thus we have proven the existence of a suitable element $l'$. Using the algorithm from Proposition \ref{anticommute} we can find an element $l'$ such that $ll'+l'l=0$ and $l'^2\in\mathbb{Q}$.

The $\mathbb{Q}$-subspace generated by $1,l,l',ll'$ is a quaternion algebra $H$ over $\mathbb{Q}$. Observe that $H\otimes\mathbb{Q}(\sqrt{d})$ has dimension 8 over $\mathbb{Q}$ and is naturally embedded into $M_2(\mathbb{Q}(\sqrt{d}))$. Hence it must be $M_2(\mathbb{Q}(\sqrt{d}))$, so $H$ is really split by $\mathbb{Q}(\sqrt{d})$.
\end{proof}

Let $\A\cong M_2(\mathbb{Q}(\sqrt{d}))$ be given by structure constants. At this point we are able construct a subalgebra $\B$ of $\A$ which is a quaternion algebra over $\Q$. This was also established in the extended abstract \cite{MACIS}. However, as $\B$ may be a division algebra, this seemingly does not help us in finding a zero divisor in $\A$.

The key observation missing from \cite{MACIS} is the following. Not every quaternion division algebra over $\Q$ can be obtained as a subalgebra of $M_2(\Q(\sqrt{d}))$, only those which are split by $\Q(\sqrt{d})$. The next theorem turns this observation into an algorithm for finding a zero divisor in $\A$:
\begin{theorem}\label{main}
Let $\A\cong M_2(\mathbb{Q}(\sqrt{d}))$ be given by structure constants. Then Algorithm \ref{alg} computes a zero divisor in $\A$. Algorithm \ref{alg} is randomized and runs in polynomial time if one is allowed to call an oracle for factoring integers.
\end{theorem}
\begin{proof}
First we construct a quaternion subalgebra $H$ over $\mathbb{Q}$ using Corollary \ref{quat}. If $H$ is isomorphic to $M_2(\mathbb{Q})$, then one can find a zero divisor in it by using the algorithm form \cite{IRS}. If not then there exists an element $s\in H$ such that $s^2=d$. Indeed, since $H$ is split by $\mathbb{Q}(\sqrt{d})$ and therefore contains $\mathbb{Q}(\sqrt{d})$ as a subfield \cite[Theorem 1.2.8]{Vigneras}. Let $1,u,v,uv$ be a quaternion basis with $u^2=a,v^2=b$. Every non-central element whose trace is zero (in $H$) is a $\mathbb{Q}$-linear combination of $u$, $v$ and $uv$. Hence finding an element $s$ such that $s^2=d$ is equivalent to solving the following equation:
\begin{equation}\label{last}
ax_1^2+bx_2^2-abx_3^2=d
\end{equation}

Since $H$ is a division algebra, the quadratic form $ax_1^2+bx_2^2-abx_3^2$ has no nontrivial zeros. Thus solving Equation \ref{last} is equivalent to finding a nontrivial zero of the quadratic form $ax_1^2+bx_2^2-abx_3^2-dx_4^2$. One can find such a zero using the algorithm from \cite{Simon}. This algorithm runs in polynomial time if one is allowed to call oracles for factoring integers. We have found an element $s$ in $H$ such that $s^2=d$. Since $H$ is a central simple algebra over $\mathbb{Q}$ and $d$ is not a square in $\mathbb{Q}$, the element $s$ is not in the center of $A$. Hence $s-\sqrt{d}$ is a zero divisor in $A$.
\end{proof}
\begin{remark}
An alternative ending of the algorithm could be the following. Assume that we have already found the subalgebra $H$. There always exists an element $s\in H$ for which $s^2=d$. We have seen this in the case where $H$ is a division algebra. If $H$ is a full matrix algebra then it is well-known. Hence the quadratic form $ax_1^2+bx_2^2-abx_3^2-dx_4^2$ is always isotropic. We find an isotropic vector $(x_1,x_2,x_3,x_4)$. If $x_4\neq 0$ we proceed as before. If $x_4=0$ then the norm of $x_1u+x_2v+x_3uv$ is 0, hence it is a zero divisor.
\end{remark}
\begin{remark}
We would like to note that Algorithm 1 only needs at most two oracle calls for integer factoring in Step 2 and 4. Furthermore, there  there are subexponential algorithms for integer factorization \cite{Fac1},\cite{Fac2} and Shor's algorithm can factor integers by a polynomial-time quantum algorithm \cite{Shor}. 
\end{remark}

First we would like to emphasize that our algorithm can be used to find nontrivial zeros of quadratic forms in three variables over $\mathbb{Q}(\sqrt{d})$ by Proposition \ref{split}. Moreover, Algorithm \ref{alg} is a reduction procedure in the following sense. The task of finding a nontrivial zero of a quadratic form in three variables over $\Q(\sqrt{d})$ can be accomplished by finding nontrivial zeros of quadratic forms in 3,4 and 6 variables over $\Q$. This reduction procedure works for any number field instead of $\Q(\sqrt{d})$. Therefore if someone finds an algorithm for finding nontrivial zeros of quadratic forms in 4 and 6 variables over $\Q(\sqrt{d})$, then one immediately has an algorithm for finding nontrivial zeros of quadratic forms in three variables over $\Q(\sqrt{d_1},\sqrt{d_2})$. The reduction procedure also works for fields of odd characteristic as demonstrated in \cite{Phd}. 

We conclude by considering the following question. The steps of Algorithm \ref{alg} perfectly make sense in the case when $\A$ is a division algebra. In that case the algorithm may fail at two points. Either it does not contain a quaternion subalgebra over $\Q$ or the subalgebra $\B$ which is a quaternion algebra over $\Q$ does not contain an element $s$ for which $s^2=d$. It is therefore natural to ask when the failure of the first type occurs, meaning the following. Assume that $\A$ contains a subalgebra $\B$ which is a quaternion algebra over $\Q$ but $\A$ is not necessarily a full matrix algebra. Does Algorithm \ref{alg} compute a quaternion subalgebra $\B$ over $\Q$? We now answer this question in the affirmative. We proceed by two facts considering the corestriction of central simple algebras. We do not define the corestriction here as it is slightly complicated and we only need certain properties of it. It is enough to note that the corestriction of a quaternion algebra over $\Q{\sqrt{d}}$ is a central simple algebra of degree 4 over $\Q$ (but as it turns out, Brauer equivalent to a quaternion algebra over $\Q$). For more details the reader is referred to \cite{Draxl},\cite{Involutions}. 

\begin{fact}\label{ARS}
Let $\Ha$ be a quaternion algebra over $\Q(\sqrt{d})$. Then $\Ha$ contains a subalgebra $\B$ which is a quaternion algebra over $\Q$ if and only if $Cor_{\Q(\sqrt{d})|\Q}(\Ha)$ (the corestriction of $\Ha$ with respect to the field extension $\Q(\sqrt{d})|\Q$) splits.
\end{fact}

The following fact is called the projection formula \cite[Part II, Theorem 7]{Draxl}:

\begin{fact}\label{projection}
Let $\Ha_{\Q(\sqrt{d})}(a,b+c\sqrt{d})$ be a quaternion algebra over $\Q(\sqrt{d})$ where $a,b,c\in\Q$. Then $Cor_{\Q(\sqrt{d})|\Q}(\Ha)$ is Brauer equivalent to $\Ha_{\Q}(a,b^2-c^2d)$.  
\end{fact}

\begin{proposition}\label{greedy}
Let $\Ha$ be a quaternion algebra over $\Q(\sqrt{d})$ which contains a quaternion subalgebra over $\Q$. Let $s\in\Ha$ such that $s^2\in\Q$. Then there exists an element $r$ such that $sr+rs=0$ and $r^2\in\Q$. 
\end{proposition}
\begin{remark}
Proposition \ref{greedy} implies that Algorithm \ref{alg} computes a quaternion subalgebra over $\Q$ even if $\Ha$ is division algebra containing a quaternion subalgebra over $\Q$. 
\end{remark}
\begin{proof}
Let $s^2=a$, where $a\in\Q$. Let $s'\in\Ha$ be such that $ss'+s's=0$ and $s'^2=b+c\sqrt{d}$. We have that $\Ha\cong\Ha_{\Q(\sqrt{d})}(a,b+c\sqrt{d})$. Proposition \ref{anticommute} says that a suitable $r$ exists if and only if $\Ha_{\Q}(a,b^2-c^2d)$ splits. So if we show that this is indeed the case then we are done. By Fact \ref{ARS} we have that $Cor_{\Q(\sqrt{d})|\Q}(\Ha)$ splits since $\Ha$ contains a quaternion subalgebra over $\Q$. By the projection formula (Fact \ref{projection}) we have that $Cor_{\Q(\sqrt{d})|\Q}(\Ha)$ is Brauer equivalent to $\Ha_{\Q}(a,b^2-c^2d)$, hence $\Ha_{\Q}(a,b^2-c^2d)$ splits. This proves the existence of a suitable element $r$.   
\end{proof}

Proposition \ref{greedy} also implies that Algorithm \ref{alg} can be used do decide if $\Ha$ contains a quaternion subalgebra over $\Q$ or not.

\paragraph*{Acknowledgement}
I would like to thank G\'abor Ivanyos and Lajos R\'onyai for their useful comments and their constant support. I am extremely grateful to the anonymous referees for the insightful remarks and suggestions. Research supported by the Hungarian National Research, Development and Innovation Office - NKFIH (Grant K115288).

\end{document}